\documentclass[12pt]{amsart}       
\usepackage[top=40mm, bottom=25mm, left=35mm, right=35mm]{geometry}
\usepackage{amsmath}
\usepackage{graphicx}
\usepackage{amssymb,latexsym,tikz}
\usepackage{mathptmx}      
\usepackage{hyperref}
\usepackage{latexsym}
\theoremstyle{plain}
\newtheorem{theorem}{Theorem}[section]
\newtheorem{corollary}[theorem]{Corollary}

\newtheorem{lemma}[theorem]{Lemma}
\newtheorem{proposition}[theorem]{Proposition}

\theoremstyle{definition}
\newtheorem{definition}[theorem]{Definition}

\theoremstyle{remark}

\newtheorem{remark}[theorem]{Remark}
\newtheorem{example}[theorem]{Example}

\newcommand{\sus}{\subseteq}

\newcommand{\pil}{\rightarrow}

\def\NN{{\mathbb N}}

\newcommand{\dist}{\text{dist}}

\definecolor{liblue}{RGB}{0,40,100}

\begin{document}

\title
[Partitions in trees and stacked simplicial complexes]
{Partitions of vertices and facets in trees and stacked simplicial complexes}

\author{Gunnar Fl{\o}ystad}
\address{Matematisk Institutt\\
         Postboks\\
         5020 Bergen}
       \email{gunnar@mi.uib.no}
       


\maketitle

\begin{abstract}
 For stacked simplicial complexes, (special subclasses of such are:
 trees, triangulations of polygons, stacked polytopes with their
 triangulations),
  we give an explicit bijection between partitions of facets (for trees: edges),
  and partitions of vertices into independent sets.
  More generally we give bijections between facet partitions whose parts
  have minimal distance $\geq s$ and vertex partitions whose parts
  have minimal distance $\geq s+1$. 
\end{abstract}

\section{Introduction}
\label{intro}

If $G = (V,E)$ is a graph with vertices $V$ and edges $E$, a set of
vertices $A \sus V$ is independent if no two vertices in $A$ are on
the same edge.
Between any two vertices there is a well-defined distance, and the vertices are
independent if their distance is $\geq 2$.
Distance may also be defined between edges in a graph.

Assume the graph $G = T$ is a tree and consider partitions of vertices in a tree
\[  V = V_0 \sqcup V_1 \sqcup \cdots \sqcup V_r \]
into $r+1$ non-empty parts. We can also partition edges
\[ E = E_1 \sqcup E_2 \sqcup \cdots \sqcup E_r \]
into $r$ non-empty parts.

\medskip
\noindent {\bf 0.}
We give a bijection
between:
\begin{itemize}
\item[i)] Partitions of vertices into $r+1$ non-empty parts,
each part consisting of independent vertices, and
\item[ii)] partitions of edges into $r$ non-empty parts (and no further
requirements).
\end{itemize}

\begin{example}
Any tree has a unique partition of the vertices into two independent sets
(two colors modulo $S_2$).
This corresponds to the partition of the edges into one part (one color).
\end{example}

Moreover the above generalizes in two directions.

\medskip
\noindent {\bf 1.}
Define a set of vertices $A \sus V$ to be $s$-scattered if the distance
between any two vertices is $\geq s$. Similarly we have the notion
of a set of edges $B \sus E$ being $s$-scattered.
We show, Theorem \ref{thm:partition-VE}, for a tree and
for $r,s \geq 1$ there is a bijection between:

\begin{itemize}
\item[i)] Partitions of vertices into $r+1$ non-empty parts, each
  part being $s+1$-scattered, and
\item[ii)] Partitions of edges into $r$ non-empty parts, each
  part being $s$-scattered
\end{itemize}

\medskip
\noindent {\bf 2.}
Stacked simplicial complexes is a generalization of trees to higher dimensions.
Stacked polytopes \cite{Grum} with the triangulation coming from a stacking of
simplices, induce a well-known subclass of stacked simplicial complexes.
The main feature of stacked simplicial complexes
for us is that between any two facets (maximal faces)
there is a unique path. Similarly between
any pair of vertices there is a unique path of facets.
We show:

\medskip
\noindent{\bf Theorem \ref{thm:partition-VE}}
{\it Let $X$ be a stacked simplicial complex of dimension $d$ and 
$r,s \geq 1$. There is a bijection between:

\begin{itemize}
\item[i)] Partitions of vertices of $X$ into $r+d$ non-empty parts, each
  part being $s+1$-scattered,
\item[ii)] Partitions of facets of $X$ into $r$ non-empty parts, each
  part being $s$-scattered
\end{itemize}
}

\medskip
The results on trees appear quite non-trivial even for the
simplest of graphs, the line graph. By considering larger and
large line graphs, we get in the limit results for partitions of
natural numbers.

\medskip
\noindent {\bf Theorem \ref{thm:nat-rd}}
{\it There is a bijection between partitions of the natural numbers $\NN$ into $r$ non-empty parts,
  each part being $s$-scattered, and partitions of $\NN$ into $r+1$ non-empty
  parts, each part being $s+1$-scattered.
}

\medskip
As a consequence we get for instance:

\medskip
\noindent {\bf Corollary \ref{cor:nat-d}}
{\it There is a bijection between partitions of $\NN$ into two parts
$A_0 \sqcup A_1$ (each part automatically $1$-scattered)
and partitions of $\NN$ into
$d+1$ parts $B_0 \sqcup \cdots \sqcup B_{d}$, each part being $d$-scattered.
}

\medskip
\noindent (A trivial special case by repeated use of Theorem \ref{thm:nat-rd},
starting from $r = s = 1$ and successively increasing $r,s$:
There is a unique partition of $\NN$ into $r$ parts, each being $r$-scattered.
These parts are of course the congruence classes modulo $r$.)

\medskip

Let us explain in more detail the bijection, first  for trees, and
secondly in an example of triangulations of polygons. Such triangulations are
stacked simplicial complexes of dimension two.

\medskip
Given a tree and a partition of the vertices $V$,
make a partition of the edges $E$ as follows:
If $v$ and $w$ are vertices consider
the unique path in $T$ linking $v$ and $w$.
Let $f$, respectively
$g$, be the edge incident to $v$, respectively $w$, on this path.
If (i)~$v$ and $w$ are in the {\it same part}~$V_i$ of~$V$ and
(ii)~{\it no other} vertex
on this path is in the part $V_i$, then put $f$
and $g$ into the same part of~$E$, and write $f \sim_E g$. The partition
of edges is the equivalence relation {\it generated} by $\sim_E$.

\medskip Conversely given a partition of the edges $E$, make a partition
of the vertices $V$ as follows:
Let $v$ and $w$ be distinct vertices, and consider
again the path from $v$ to $w$. Let the edges $f,g$ be as above.
If (i)~the edges $f$ and $g$ are distinct (equivalently the vertices
$v,w$ are independent),
(ii)~$f$ and $g$ are in the {\it same
part} $E_j$, and (iii)~{\it no other} edge on this path is in the part $E_j$,
then put $v$ and $w$ in the same part of $V$, and write
$v \sim_V w$. The partition of vertices is the equivalence relation
generated by $\sim_V$.



\begin{example}\label{eks:intro-graph}
  In Figure~\ref{figfive} 
  we partition the edges into two parts, colored red and black. The vertices
  are then partitioned into three parts, each consisting of independent vertices. 
  The partition of the vertex set of the first tree is
  \[ \{1,3,5\}\cup\{2,6\}\cup\{4\},\]
  and that of the second tree is
  \[\{1,3,5,8,10\}\cup\{2,4,7\}\cup\{6,9\}.\]
\end{example}

\begin{figure}
\begin{center}
\begin{tikzpicture}[dot/.style={draw,fill,circle,inner sep=1pt},scale=1.2]
\draw [help lines,white] (1,-.5) grid (6,0);
\draw[very thick,red] (3,0)--(5,0);
\draw[very thick] (1,0)--(3,0);
\draw[very thick] (5,0)--(6,0);
\fill (1,0)  circle (0.09);
\fill (2,0)  circle (0.09);
\fill (3,0)  circle (0.09);
\fill (4,0)  circle (0.09);
\fill (5,0)  circle (0.09);
\fill (6,0)  circle (0.09);
\fill[yellow] (1,0)  circle (0.065);
\fill[green] (2,0)  circle (0.065);
\fill[yellow] (3,0)  circle (0.065);
\fill[blue] (4,0)  circle (0.065);
\fill[yellow] (5,0)  circle (0.065);
\fill[green] (6,0)  circle (0.065);
\node at (1,.25) {$1$};
\node at (2,.25) {$2$};
\node at (3,.25) {$3$};
\node at (4,.25) {$4$};
\node at (5,.25) {$5$};
\node at (6,.25) {$6$};
\end{tikzpicture}\\
\begin{tikzpicture}[dot/.style={draw,fill,circle,inner sep=1pt},scale=1.1]
\draw[very thick,red] (5,0)--(6,0);
\draw[very thick] (1,0)--(5,0);
\draw[very thick] (6,0)--(7,0);
\draw[very thick] (4,0)--(4,-1);
\draw[very thick,red] (4,-1)--(4,-3);
\fill (1,0)  circle (0.09);
\fill (2,0)  circle (0.09);
\fill (3,0)  circle (0.09);
\fill (4,0)  circle (0.09);
\fill (5,0)  circle (0.09);
\fill (6,0)  circle (0.09);
\fill (7,0)  circle (0.09);
\fill (4,-1)  circle (0.09);
\fill (4,-2)  circle (0.09);
\fill (4,-3)  circle (0.09);
\fill[yellow] (1,0)  circle (0.065);
\fill[green] (2,0)  circle (0.065);
\fill[yellow] (3,0)  circle (0.065);
\fill[green] (4,0)  circle (0.065);
\fill[yellow] (5,0)  circle (0.065);
\fill[blue] (6,0)  circle (0.065);
\fill[green] (7,0)  circle (0.065);
\fill[yellow] (4,-1)  circle (0.065);
\fill[blue] (4,-2)  circle (0.065);
\fill[yellow] (4,-3)  circle (0.065);
\node at (1,.25) {$1$};
\node at (2,.25) {$2$};
\node at (3,.25) {$3$};
\node at (4,.25) {$4$};
\node at (5,.25) {$5$};
\node at (6,.25) {$6$};
\node at (7,.25) {$7$};
\node at (4.22,-1) {$8$};
\node at (4.22,-2) {$9$};
\node at (4.3,-3) {$10$};
\end{tikzpicture}
\caption{Partitions of edges into two parts and corresponding
partitions of vertices into three parts}
\label{figfive}
\end{center}
\end{figure}
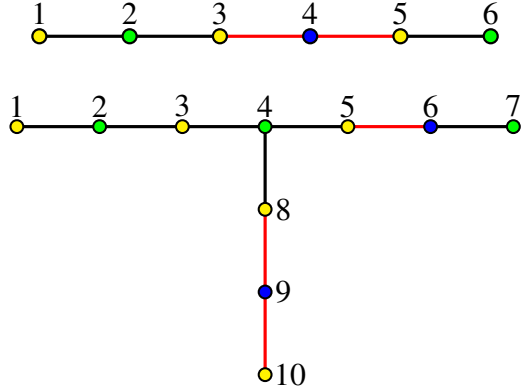

\medskip

\begin{example}
  Consider now a triangulation of the heptagon, which is
  a stacked simplicial complex. We partition the
  facets into two parts: three blue and two red, Figure \ref{figpol}.
  This corresponds, in a similar way as for trees, to a partition of the
  vertices, now into four parts of independent vertices: two yellow
  vertices $3,7$, three red $1,4,6$, one blue $2$, and one green $5$.
  In fact the facet parts here are $2$-scattered and so each of the
  vertex parts will even be $3$-scattered.

  Note that the colors have no real significance here, they are just
  added for pedagogical and visualisation purposes.
  In particular there is no connection
  between colors of facets and colors of vertices.
  \end{example}

\begin{figure}
\begin{center}
  \begin{tikzpicture}[dot/.style={draw,fill,circle,inner sep=2.3pt},scale=2]
\draw [help lines, white] (-1.5,-1.5) grid (1.5,1.2);


\draw [fill=blue, opacity=0.3] (0,1)--(.79,.63)--(-.79,.63)--(0,1);
\draw [fill=red, opacity=0.3] (.79,.63)--(-.43,-.9)--(-.79,.63)--(.79,.63);
\draw [fill=blue, opacity=0.3] (.79,.63)--(.43,-.9)--(-.43,-.9)--(.79,.63);
\draw [fill=blue, opacity=0.3] (-.43,-.9)--(-.79,.63)--(-.98,-.23)--(-.43,-.9);
\draw [fill=red, opacity=0.3] (.79,.63)--(.98,-.23)--(.43,-.9)--(.79,.63);

  \foreach \l [count=\n] in {0,1,2,3,4,5,6} {
    \pgfmathsetmacro\angle{90-360/7*(\n-1)}
      \node[dot] (n\n) at (\angle:1) {};
      \fill (\angle:1)  circle (0.09);
      }

\fill[red] (90:1) circle (0.065);
\fill[yellow] (141.43:1) circle (0.065);
\fill[red] (192.86:1) circle (0.065);
\fill[green] (244.29:1) circle (0.065);
\fill[red] (295.72:1) circle (0.065);
\fill[yellow] (347.15:1) circle (0.065);
\fill[blue] (398.58:1) circle (0.065);
      
  \foreach \l [count=\n] in {0,1,2,3,4,5,6} {
    \pgfmathsetmacro\angle{90-360/7*(\n-1)}
      \node (m\n) at (\angle:1.3) {$\n$};
  }
  \draw[thick] (n1) -- (n2) -- (n3) -- (n4) -- (n5) -- (n6) -- (n7) -- (n1);
  \draw[thick] (n2)--(n7)--(n5)--(n2)--(n4);
\end{tikzpicture}
\caption{Partition of facets into two parts and corresponding partition
of vertices into four parts}
\label{figpol}
\end{center}
\end{figure}
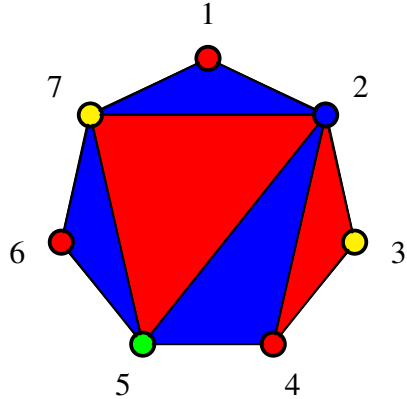

Our results here came out of work in \cite{FlOr} by M.Orlich and the author
on triangulations of polygons and more generally on stacked
simplicial complexes.

Results related to the present article are to be found  in
enumerative combinatorics. A first result in this direction is
W.Yang \cite{Ya} considering trees on $n+1$ vertices,
showing that partitions into independent sets of vertices are counted by Bell
numbers $B_{n}$. He also considers generalized $d$-trees and show
that partitions into independent sets of vertices of such a tree with
$n+d$ vertices is counted by the
Bell number $B_n$. A generalized $d$-tree is the same as the edges
(or $1$-skeleton) of a stacked simplicial complex of dimension $d$. 

B.Duncan and R.Peele \cite{DuPe} consider enumerative aspects of partitions
of vertices of graphs into independent sets. For trees they
show that partitions of a tree with $n+1$ vertices into $k+1$ independent
parts are in bijection with partitions of an $n$-set into $k$ parts.
They give, mostly illustrated by a large example,
a bijection which is essentially the same as we give.
But we gain here the conceptual advantage of expressing this in terms
of edges of the tree. Their statement involves choosing an arbitrary
vertex $r$, a root, and then relating independent vertex partitions
of $V$ to partitions of $V \backslash \{r\}$.
Their less conceptual form may also be the reason they do not
have the extension to $s$-scattered parts.
\cite{HeMe} and \cite{KeNy} considers enumerative aspects of
graphical Bell numbers further, and 
the latter also generalized $d$-trees. 

W.Chen, E.Deng, and R.Du  \cite{CDD} consider ordered sets and use the
terminology $m$-regular for $m$-scattered. They show that partitions
of an ordered $n+1$-set into $k+1$ parts which are $m+1$-regular
are in bijection with partitions of an $n$-set into $k$ parts
which are $m$-regular. This corresponds to the case of line
graphs, and for this case the bijection they give is essentially the
same as ours. We discuss this in Section \ref{sec:nat}. They also
show that we have bijections in this case when considering non-crossing
partitions. Related enumerative results are found in \cite{Mu}
and \cite{CW}. The book \cite{TM} is a comprehensive account of
partitions of ordered sets. 

\begin{remark}
  The results of this paper dropped out of investigations
  in \cite{FlOr}, concerning Stanley-Reisner rings of stacked simplicial
  complexes. Among such simplicial complexes there are certain
  {\it separated models}, and from these models we obtain every
  stacked simplicial complex by partitioning the vertices into
  independent classes and then collapsing each class into a single vertex.
  This is done such that the essential algebraic and homological properties of
  the associated rings are preserved.
  
  The algebra involved here should generalize to much larger classes
  of simplicial complexes. Polarizing Artin monomial ideals gives
  separated models, and in \cite{AFL} we describe all such for
  polarizations of any power of a  graded maximal ideal in a polynomial
  ring. The case
  of stacked simplicial complex is the case of the {\it second} power.
  Polarizing Artin monomial ideals in general still goes much further
  than \cite{AFL}. The results
  in the present article therefore likely have vast generalizations,
  see Subsection 8.2 of \cite{FlMa}.
  This should involve partitions of vertices, but it is a challenge
  what other ingredients and statements should be involved.

  Let us also mention that the main result of \cite{FlMa} is a
  fundamental theorem of combinatorial geometry of monomial ideals.
  Namely that any polarization of an Artin monomial ideal, via the
  Stanley-Reisner correspondence, is a simplicial complex whose
  topological realization is a ball. 
\end{remark}

We describe the  organization of this article. Section \ref{sec:stacked} gives
the notion of stacked simplicial complex of dimension $d$.
It characterizes these, Proposition \ref{pro:stacked-ekviv},
as the simplicial complexes for which there is a unique path between any
pair of facets, and the number of vertices is $d$ more
then that number of facets. Section \ref{sec:main} describes the
correspondence between partitions of vertices and partitions of facets,
and shows our main Theorem \ref{thm:partition-VE}.
In the last Section \ref{sec:nat} we specialize these results to bijections
between partitions  of natural numbers, the parts having lower bound
requirements on minimal distance.

\noindent {\it Acknowledgements.}
 I thank M.Orlich for making
  Figure \ref{figfive}.

\section{Stacked simplicial complexes}
\label{sec:stacked}

We recall the notion of stacked simplicial complex.
Its main feature from our perspective, is that it
generalizes the property of trees, that for every two
faces, there is a unique path between them. 

\subsection{Paths}
Let $V$ be a finite set. A {\it simplicial complex} $X$ on $V$
is a family of subsets of $V$ 
closed under taking subsets of each element of the family. 
So for an element $F \in X$ and $G \sus F$, then also  $G \in X$.
The elements of $V$ are {\it vertices}, the elemens of $X$ are {\it faces},
and the maximal elements of $X$ for the inclusion relation are
{\it facets}.
A simplicial complex has a natural geometric realization. A
face with $d$ vertices is then realized as a simplex of dimension $d-1$. 

Given {\it any} family $Y$ of subsets of $V$, the simplicial complex 
{\it generated
by $Y$} is the family of all subsets $G$ of $V$ such that $G \sus F$ for
some $F \in Y$. 


\begin{definition}
  A pure simplicial complex (i.e., where all the facets have the same dimension)
  is {\it stacked} if there is an ordering of its
  facets $F_0, F_1, \ldots, F_k$ such that if $X_{p-1}$ is the simplicial
  complex generated by $F_0, \ldots, F_{p-1}$, then for $p \geq 1$ the facet
  $F_p$ is attached to
  $X_{p-1}$ along a single codimension one face of $F_p$. So we may write
  $F_p = G_p \cup \{v_p \}$ where $G_p$ is a codimension one face of $X_{p-1}$
  and $v_p$ is not a vertex of $X_{p-1}$. The vertex $v_p$ is the
  {\it free vertex} of $F_p$, in this stacking order.
\end{definition}

\begin{remark}
This is a special case of shellable simplicial complexes,
see~\cite[Subsection~8.2]{HeHi}. It is not the same as the notion of
simplicial complex being a tree as in~\cite{Far}, even if the tree
is pure. Rather the notion of stacked simplicial complex
is more general. For instance the triangulation of the heptagon
given in Figure \ref{figpol}, is not a tree in the
sense of \cite{Far}, since removing the triangles~$234$ and~$257$
one has no facet which is a leaf.
\end{remark}

\medskip
\begin{definition}
  A {\it (gallery) walk} in a pure simplicial complex, is a sequence
  of facets $f_1, \ldots, f_p$ such that each
  $g_i = f_i \cap f_{i+1}$ has codimension one in $f_i$
  (and hence also in $f_{i+1}$).
The {\it left end vertex} is the single element of $f_1 \backslash f_2$
and the {\it right end vertex} is the single element of $f_p \backslash f_{p-1}$.
\end{definition}

If $f_i \cap f_{i+1} = f_j \cap f_{j+1}$ for some $1 \leq i < j < p$,
we can make a shorter walk from $f_1$ to $f_p$.
If $f_i \neq f_{j+1}$ then
\[ f_1, \ldots, f_i, f_{j+1}, \ldots, f_p\]
is a shorter walk. If $f_i = f_{j+1}$, then
\[ f_1, \ldots, f_{i-1}, f_{j+1}, \ldots, f_p \]
is a shorter walk from $f_1$ to $f_p$. 

\begin{definition} \label{def:stacked-path}
  A {\it path} is a walk $f_1,f_2, \ldots, f_p$
  where all the $f_i \cap f_{i+1}$ are distinct. The {\it length} of the path is
  $p-1$.  
\end{definition}

By the explanation before this definition,
any walk from $f_1$ to $f_p$ may be reduced to a path between $f_1$ and $f_p$.

\begin{lemma} \label{lem:stacked-order}
  Let $X$ be a stacked simplicial complex and $f_1,f_2, \ldots, f_p$ a path
  in $X$.
  \begin{itemize}
  \item[a.] Let $f_i$ come last among the facets on the path,
    for a stacking order of $X$, and let $v$ be its free vertex.
    Then either $f_i = f_1$ and
  $\{v\} = f_1\backslash f_2$ or $f_i = f_p$ and
  $\{v \} = f_p \backslash f_{p-1}$.
  \item[b.] All the $f_i$ are distinct.
  \end{itemize}
\end{lemma}

\begin{proof}
\noindent a.  If $1 < i < p$, then $f_i = (f_i \cap f_{i-1}) \cup (f_i \cap f_{i+1})$,
  and so $v$ would be on two facets. But this is not so.
  So $f_i$ must be one of the end vertices and $\{v\}$ must be as above.

\noindent  b. If $f_1 = f_p$ then $p \geq 3$ and as $v$ is on only one facet on
  the path, $f_1 \cap f_2 = f_1 \backslash \{ v \} = f_p \backslash \{v \}
  = f_{p-1} \cap f_p$, contradicting that we have a path.

  If $1 \leq j < k \leq p$, then $f_j, f_{j+1}, \ldots, f_k$ is also a path
  and so $f_j$ and $f_k$ must be distinct.  \qed
\end{proof}


\begin{definition} Let $X$ be a pure simplicial complex, and
  $h,k$ faces in $X$ such that $h \cup k$ is not contained
   in any codimension one face.
   A {\it path between $h$ and $k$}, written
   \[ h | f_1, \ldots, f_p | k\]
   is a path $f_1, \ldots, f_p$ such that
   \[ h \sus f_1,\, h \not \sus f_1 \cap f_2, \quad  k \sus f_p, \,
     k \not \sus f_p \cap f_{p-1}. \]
   The {\it face-distance} (or simply {\it distance}) between
   $h$ and $k$ is $p$.
   In particular, if $h,k$ are contained in a common facet, and
   not in a codimension one face, their distance is one. 
   If $h \cup k$ is contained in a codimension one face, a path
   between is an empty path consisting of no facets, written $h || k$.
   Their distance is defined to be zero. 
\end{definition}

We mostly use this definition when $h$ and $k$ are single vertex sets $\{v\}$
and $\{w\}$, in which case we simply
write $v$ instead of $\{v\}$,
and similarly for $w$.

\begin{lemma} 
  \label{lem:stacked-hk} Let $X$ be a stacked simplicial complex, and
  $h,k$ faces of $X$.
  \begin{itemize}
  \item[a.]
    In a path $f_1, \ldots, f_p$
  between $h$ and $k$ we have $h \not \sus f_i \cap f_{i+1} $ and
  $k \not \sus f_i \cap f_{i+1}$ for each $1 \leq i < p$.
\item[b.] Let $f$ be the facet on the path which is last for a stacking order
  on $X$. 
  The free vertex of $f$ is in $h$ or in $k$.
\end{itemize}

\end{lemma}

\begin{proof}
  \noindent a.  Suppose for $h$ there was an $r$ such that
  $h \sus f_r \cap f_{r+1}$, and let $r$ be minimal such.
Then $r \geq 2$ and consider the path
  \[ f_1, \ldots, f_r \]
  Let $f$ be the last among these facets in the stacking order of $X$, and
  $v$ its free vertex. Then $f = f_1$ or $f = f_r$.
  If $f = f_1$, since $h \sus f_1$ and $h \not \sus f_1 \cap f_2$, $h$ must
  contain the single vertex in $f_1 \backslash f_1 \cap f_2$, and this
  vertex is $v$. But then $v$ is also in $f_r$. Since $v$ is only on
  a single facet, then $f_r = f_1$, and this contradicts all facets on
  a path being distinct.
  
  \noindent b. The last facet is one of the end facets, say $f_1$. If
  $p = 1$ then $f_1 = h \cup k$ and the statement holds. If $p \geq 2$, the
  $f_1 = h \cup (f_1 \cap f_2)$ and since $v$ is not on two facets in the
  path, we get $v \in h$. \qed
\end{proof}

\subsection{Existence and uniqueness of paths}

  \begin{proposition} \label{prop:stacked-path}
    Let $h,k$ be faces with union $h \cup k$ not contained in any
    codimension one face. Then there is a unique path between them.
  \end{proposition}

  \begin{proof}
{\it Existence:}
    If $h \cup k$ is a facet $f$, then $h | f | k$ is a
    path between them. Assume then $h \cup k$ is not contained in a facet.

    Let $f$ be a facet containing $h$ and $f^\prime$
    a facet containing $k$. Let
    \[ f = f_1, \ldots, f_p = f^\prime \]
    be the path between them. Let $i$ be maximal such that $h \sus f_i$.
    Let $j \geq i$ be minimal such that $k \sus f_j$.
    Then we have a path from $h$ to $k$:
    \[ h | f_i, \ldots, f_j | k. \]

    \medskip
\noindent    {\it Uniqueness:}
     Let
    \[ h | f_1, \ldots, f_p | k, \quad h | f_1^\prime, \ldots , f_q^\prime | k \]
    be paths with $p,q \geq 1$. 

    Let $f$ be the last of the facets in these paths,
    in the stacking order,
    and with free vertex $v$.
    By Lemma \ref{lem:stacked-hk}, $v$ is in say $h$. As $v$ is in a single
    facet on these paths, we get $f = f_1^\prime = f_1$.
    If $p = q = 1$ we are done.

\noindent i) Suppose exactly one of $p,q$ is  $\geq 2$, say $p = 1$ (and $q \geq 2$).
    Then $f_1^\prime = f_1 = h \cup k$ so $k \sus f_1^\prime$.
    Since $k \not \sus f_1^\prime
    \cap f_2^\prime$ the free vertex $v$ is also in $k$, and so in $f_q^\prime$.
    Then $f_q^\prime = f = f_1^\prime$, which is not so for a path.   
    

\noindent ii) Suppose the paths have $p,q \geq 2$. 
    Let
    \[ g := f \backslash \{v\} = f_1 \cap f_2 = f_1^\prime \cap f_2^\prime. \]
    Since $k \not \sus f_1 \cap f_2 = g$,
    we have $g \cup k$ not included in a codimension
    one face. So we get paths
    \[ g | f_2, \ldots, f_p | k, \quad g | f_2^\prime, \ldots, f_q^\prime | k. \]
    By induction on length, these paths are equal.   \qed  
  \end{proof}


The following generalizes the well-known
situation for trees.

\begin{proposition} \label{pro:stacked-ekviv}
  Let $X$ be a pure simplicial complex of dimension $d$
with $n$ facets and $v$ vertices.  Then $X$ is stacked iff $v = n+d$ and 
between every pair of facets there is a unique path.
\end{proposition}

\begin{proof}
  When $X$ is stacked it is clear from construction that $v = n +d$.
  By Proposition \ref{prop:stacked-path} there is a unique path
  between any two facets. 

\medskip
Conversely, assume $v = n+d$ and that between every pair of facets
there is a unique path. Choose a facet $f$, order the
facets of $X$:
\begin{equation} \label{eq:stacked-dist}
  f = f_1, f_2, \ldots, f_{n}  \end{equation}
such that the distance $\dist(f,f_i) \leq \dist(f,f_j)$ for $i \leq j$. 
Let $Y_p$ be the simplicial complex generated by $f_1, \ldots, f_p$.
Consider the path from $f$ to $f_{p+1}$:
\[f = f^1, f^2, \ldots, f^r = f_{p+1}. \]
Here all facets except the last $f^r$ are in $Y_p$ as they must be listed
before $f_{p+1}$ in \eqref{eq:stacked-dist} due to distance.
Since $f^{r-1} \cap f^r$ has codimension one in $f_{p+1}$, when passing
from $Y_p$ to $Y_{p+1}$ we have added at most one new vertex.
Since $Y_n =X$ has $d+n$ vertices, we must have added exactly one
new vertex each time, and so $f_n$ has a single free vertex, and
$Y_{n-1}$ has $d+n-1$ vertices.

We show that between any two facets of $Y_{n-1}$ there is a unique
path. By induction $f_1, \ldots, f_{n-1}$ is a stacking order for $Y_{n-1}$,
and so \eqref{eq:stacked-dist} gives that $Y_n$ is also stacked.

Let $f_r,f_s$ be two facets in $Y_{n-1}$. Consider the path
$f_r = f_1^\prime, \ldots, f_p^\prime = f_s$ in $X = Y_n$.
If the last facet $f_n$ is on this path, say  $f_n= f_t^\prime$,
then $1 < t < p$ and
\[ f_n = f_t^\prime  = (f_t^\prime \cap f_{t-1}^\prime) \cup (f_t^\prime \cap _{t+1})
  \sus Y_{n-1}. \]
This is not so, and so the path from $f_r$ to $f_s$ is entirely in
$Y_{n-1}$. \qed
\end{proof}

\noindent The {\it facet-distance} between two facets $f, g$
is defined to be the
length of the path $f|f_1, \ldots, f_p|g$ between them (note that
$f = f_1$ and $g = f_p$), and this length is $p-1$.

\begin{remark}
  Note that for two facets $f,g$ their face-distance is one more than
  their facet-distance.
It may seem awkward to have two notions of distance. But
  by looking at the graph:
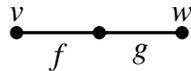
\[\begin{tikzpicture}[dot/.style={draw,fill,circle,inner sep=1pt},scale=1.1]
\fill (0,0)  circle (.08);
\fill (1,0)  circle (.08);
\fill (2,0)  circle (.08);
\draw[very thick] (0,0)--(1.1,0);
\draw[very thick] (1,0)--(2,0);
\coordinate [label=above: $v$] (v) at (0,0);
\coordinate [label=below: $f$] (f) at (.5,0);
\coordinate [label=below: $g$] (g) at (1.5,0);
\coordinate [label=above: $w$] (w) at (2,0);
\end{tikzpicture}\]
it is natural. 
  The facet-distance between $f$ and $g$ is one, and the face-distance
  between $v$ and $w$ is two. (And the face-distance between $f$ and $g$
  is also two.)
\end{remark}


\subsection{Distance neighborhoods}
Let $X$ be a pure simplicial complex. 
  Choose a codimension one face $g$ in $X$.
  For $m \geq 1$, let $X_m$ be the simplicial complex generated by
  those facets of $X$ whose face-distance to
  $g$ is $\leq m$. In particular $X_0 = \emptyset$ and
  the facets of $X_1$ are the facets of $X$ containing
  $g$. Let $V_0$ be the vertices of $g$ and $V_m$ the vertices of $X_m$
  for $m \geq 1$.

  \begin{lemma} \label{lem:stacked-vVp} Assume $X$ is a stacked simplicial
    complex. 
 For $m \geq 1$, if $v \in V_m \backslash V_{m-1}$, there is a unique
  facet $f_v$ in $X_m$ containing $v$, and $f_v \backslash \{v\}$
  is a subset of $V_{m-1}$.
\end{lemma}

\begin{proof}
  Let $f_1$ be a facet in $X_m$ containing $v$, and
  \[ f_1 | f_1, \ldots, f_r | g \]
  the unique path from $f_1$ to $g$. We have $r \leq m$.
  Let $s$ be maximal with $v \in f_s$. Then in $X_m$:
  \[ v | f_s, \ldots, f_r | g \] is the unique path from $v$ to $g$.
  Since $v \in V_{m}\backslash V_{m-1}$ we must have $r-s+1 = m$,
  and so $r = m$ and $s = 1$.
  Then $f_1$ is the first facet in the unique path from $v$ to $g$ and
  $f_1 \backslash \{ v \} \sus f_2 \sus V_{m-1}$. \qed
\end{proof}

\begin{corollary} $X_m$ is a stacked simplicial complex on $V_m$.
\end{corollary}

\begin{proof} Let $f_1, \ldots, f_r$ be any ordering of the facets
  in $X_m$ such that the face-distance between $g$ and $f_i$ is
  weakly increasing with $i$. Choose $1 \leq k \leq r$
  and let $d$ be the distance from $g$ to $f_k$.
  The facets $f_1, \ldots, f_k$
  are then all in $X_d$. Let $f_1^\prime, \ldots, f_d^\prime = f_k$ be the
  unique path from $g$ to $f_k$. Then:
  \begin{itemize}
    \item If $d \geq 2$ then $f_{d-1}^\prime = f_i$ for some $i < k$,
    \item $f_k = f_v$ for some $v \in V_d \backslash V_{d-1}$,
    \item $f_i \cap f_k = f_k \backslash \{v\}$,
    \item By Lemma \ref{lem:stacked-vVp}, $v$ is in none of the $f_i$
      for $i < k$.
    \end{itemize}
    Thus $f_1, \ldots, f_r$ is a stacking order for $X_m$. \qed
  \end{proof}

\section{Bijections between partitions of facets and of
  vertices}  \label{sec:main}

We show how partitions of facets and partitions of vertices into
independent sets correspond, and we show that this correspondence
is really a bijection. Our arguments are by induction on the
distance neighbourhood $X_m$. We develop some lemmata before
the proof of the main theorem. 

\subsection{Bijections between partitions}

\begin{definition}
Let $X$ be a be a stacked simplicial complex with vertex set $V$, 
and $s$ an integer $\geq 1$.
A subset $A \sus V$ is {\it $s$-scattered}
if the face-distance between any two distinct vertices in $A$ is
$\geq s$. The vertex set is {\it independent} if it is $2$-scattered,
i.e. no two vertices in $A$ are on the same facet.

Similarly a subset $B$ of the facets is $s$-scattered if the
facet-distance between any two facets in $B$ is $\geq s$. 
\end{definition}

We consider partitions of the vertices
\begin{equation} \label{eq:part-V}
  V = V_1 \sqcup V_2 \sqcup \cdots \sqcup V_r
\end{equation}
into non-empty disjoint sets. Note that the  order here is not relevant, so
if we switch $V_i$ and $V_j$ we have the same partition.

\begin{remark} If the $V_i$ are independent,
  this is almost the same as a graph coloring of vertices, but not quite.
  A coloring of the vertices $V$ is a map $f : V \pil C$ where
  $C = \{ c_1, \ldots, c_r\}$ is a set of colors, such that
  each inverse image $f^{-1}(c_i)$ is a set of independent vertices. 
  The symmetric group $S_r$ acts on colorings by permuting the colors
  $c_1, \ldots, c_r$.
  So a partition as above \eqref{eq:part-V}
is an orbit for the action of $S_r$. The class of such orbits, or equivalently
of partitions \eqref{eq:part-V} are also called
{\it non-equivalent vertex colorings}, see \cite{HeMe}. 
\end{remark}

We also consider partitions of the facets
\[ F = F_1 \sqcup F_2 \sqcup \cdots \sqcup F_s. \]

\medskip
\subsubsection{From vertex partitions to facet partitions.}
\label{subsec:partition-vf}
Now we make a correspondence as follows. Given a partition of $V$
into non-empty independent sets, given
by an equivalence relation $\sim_V$. For ease of following the
arguments, we will think of each part as having a specific color.
Make a partition of $F$ as follows. Let $f$ and $f^\prime$ be distinct
facets. Consider the unique path in $X$ between them:
\[ f =  f_1, \ldots, f_p = f^\prime, \]
and let $v$ and $w$ be respectively the left and right end vertices. 
If $v$ and $w$ have the same color, say blue, and none of
the facets $f_2, \ldots, f_{p-1}$ has any blue vertex, then
write $f \sim^\prime_F f^\prime$. This means that $f$ and $f^\prime$
will be in the same part of facets, these facets get the same
color. The colors of vertices and edges are however unrelated,
so the facets $f$ and $f^\prime$ get some color unrelated to blue.
The relation $\sim_F$ on $F$ is the equivalence relation
generated by $\sim_F^\prime$.

\medskip
\subsubsection{From facet partitions to vertex partitions}
\label{subsec:partition-fv}
Conversely given a partition of the facets $F$, given
by an equivalence relation $\sim_F$. Make a partition
of the vertices $V$ as follows.
Let $v$ and $w$ be independent vertices, and consider
the path from $v$ to $w$:
\[ v | f_1, \ldots, f_p | w. \]
If $f_1$ and $f_p$ have the same color, say green, and none of the
facets $f_2, \ldots, f_{p-1}$ are green, then let $v \sim^\prime_V w$.
The relation $\sim_V$ is the
equivalence relation generated by $\sim^\prime_V$. See Example
\ref{eks:intro-graph}.

\medskip
We want to show that these correspondences are inverse to each other.
To do this we show:

\medskip
\noindent {\bf A.} From an equivalence relation $\sim_F$ on $F$, we have
constructed the equivalence relation $\sim_V$ on $V$.
We show that the equivalence relation $\sim_V$ in turn induces
the equivalence relation $\sim_F$ by showing:
\begin{itemize}
\item[1.] If $v \sim_V w$ are distinct and, say blue, and
\[ v | f_1, \ldots, f_p | w \] where none of $f_2, \ldots, f_{p-1}$ have
a blue vertex, then $p \geq 2$ and $f_1 \sim_F f_p$
(in the original equivalence
relation for $F$)
\item[2.] If $f \sim_F g$ in the original relation with $f,g$ distinct,
 there is a sequence
  $f= f_0, f_1, \ldots, f_p = g$ such that
  \begin{equation} \label{eq:f-sekvens}
    f_0 \sim_F^\prime f_1 \sim_F^\prime  \cdots
    \sim_F^\prime f_p
    \end{equation} 
where $\sim^\prime_F$ is the relation constructed from $\sim_V$.
\item[3.] We also show that if the original $\sim_F$ is
  $s$-scattered, then $\sim_V$ is $s+1$-scattered.
  \end{itemize}

\medskip
\noindent
{\bf B.} From an equivalence relation $\sim_V$ on $V$, we have
constructed the equivalence relation $\sim_F$ on $F$. We show that this
in turn induces the equivalence relation $\sim_V$, by showing:
\begin{itemize}
\item[1.] If we have a path with $p \geq 2$
\[ v | f_1, \ldots, f_p | w \] where $f_1$ and $f_p$ are, say green,
and none of $f_2, \ldots, f_{p-1}$ are green, 
then $v \sim_V w$ (in the original equivalence
relation for $V$).
\item[2.] If $v \sim_V w$ in the original relation, there is
  a sequence $v = v_0, v_1, \ldots , v_p = w$ such that
  \begin{equation} \label{eq:v-sekvens}
    v_0 \sim_V^\prime v_1 \sim_V^\prime  \cdots
    \sim_V^\prime v_p,
  \end{equation}
where $\sim_V^\prime$ is the relation constructed from $\sim_F$.
\item[3.] We also show that if the original $\sim_V$ is  $s+1$-scattered,
  then $\sim_F$ is $s$-scattered.
\end{itemize}

\subsection{Induction arguments on $X_m$}
We show {\bf A, B} for the $X_m$ by induction on $m$. For this we need some
lemmata. 

\begin{lemma} \label{lem:partition-usimw} Given a partition of the facets $F$,
  and consider the relation $\sim_V^\prime$ on vertices, constructed above
  in Subsection \ref{subsec:partition-fv}.
  Let $v \in V_m \backslash V_{m-1}$, and $u,w$ distinct in $V_m$.
  If $v \sim_V^\prime u$ and $v \sim_V^\prime w$, then 
  $u \sim_V^\prime w$.
\end{lemma}

\begin{proof}
  We show first that $u,w$ are independent. Recall, Lemma \ref{lem:stacked-vVp},
  that $f_v$ is the unique facet on $X_m$ containing $v$.

  \noindent i) Suppose $u,w$ were on the same
  face $f$. Let
  \[ f_v = f_1, \ldots, f_p = f \]
  be the path from $f_v$ to $f$. 
  If $u,w$ are both on $f_{p-1}$ we may
  make a shorter path. So we may assume $u,w \in f_p$ and not both
  in $f_p \cap f_{p-1}$. Assume then that $u$ is the right end vertex of
  $f_p$. Then the above must be the unique path from $v$ to $u$.
  Since $v \sim_V^\prime u$, $f_1$ and $f_p$ have the same color, say green.
  We have $w \in f_{p-1}$ and $w \not \in f_1$ (since $v,w$ are independent).
  Let $r \geq 2$ be minimal
  such that $w \in f_r$. We then have a path
  \[ v | f_1, \ldots, f_r | w \] and this is the unique path from $v$ to
  $w$. By definition of $\sim_V^\prime$, $f_1$ and $f_r$ also have the same color,
  which must be green.
  But by definition of
  $v \sim_V^\prime u$ there should not have been any green color
  between the faces $f_1$ and $f_p$. Hence $u,w$ must be independent.

  \medskip
  \noindent ii) Let 
  \begin{equation} \label{eq:partition-vuw}
    v | f_1, \ldots, f_p | u, \quad v | f^\prime_1, \ldots, f^\prime_q | w,
  \end{equation}
  be the unique paths where $f_1 = f_1^\prime$ by Lemma \ref{lem:stacked-vVp}.
  Here $f_1$ and $f_p$ have the same color, say green,
  and $f_1^\prime (= f_1)$ and $f_q^\prime$ have the same color, also green,
  and none of the facets in between have color green.
  None of the two paths is then a subpath of the other. Hence there is
  an $r$ such that $f_r \neq f_r^\prime$, and let $r$ be minimal such, so
  $r \geq 2$. If $r \geq 3$ there is a walk
  \begin{equation}  \label{eq:partition-splice}
    u | f_p, \ldots, f_r, f_{r-1} = f_{r-1}^\prime, f_r^\prime,
    \ldots, f_q^\prime | w
    \end{equation}
    where $f_p$ and $f_q^\prime$ are the only green facets. If $r = 2$
    then $f_2, f_2^\prime \supseteq  f_1 \backslash \{v\}$ and so
    $f_2 \cap f_2^\prime$ has codimension one. There is then a walk
\begin{equation} \label{eq:partition-splice2}
    u | f_p, \ldots, f_2, f_2^\prime,
    \ldots, f_q^\prime | w
    \end{equation}
where only the end facets are green. 
By reducing these walks like before Definition \ref{def:stacked-path},
we get a path giving $u \sim_V^\prime w$. \qed
\end{proof}

\noindent {\it Note.} The process of suitably cutting the sequences
\eqref{eq:partition-vuw}, then splicing them, \eqref{eq:partition-splice}
or \eqref{eq:partition-splice2}, and lastly reducing to a path,
will be used a couple
of times and we call it {\it cut-splice-reduction}.

\begin{lemma} \label{lem:partition-gsimh} Given a partition of $V$
  into independent sets, and the relation $\sim_F^\prime$ constructed
  as in Subsection \ref{subsec:partition-vf}.
  Let $f$ be a facet in $X_m$ which is not in $X_{m-1}$.
  If $f  \sim^\prime_F g$ and $f \sim^\prime_F h$, then $g \sim^\prime_F h$.
\end{lemma}

\begin{proof}
  Note first that $f$ is $f_v$ for a unique $v$. By Lemma
  \ref{lem:stacked-vVp} any path in $X_m$ starting from $f$ must
  have left end vertex $v$. 
  So we have the following paths
  \[ v | f = f_1, \ldots, f_p = g | u, \quad v | f = f_1^\prime, \ldots,
    f_q^\prime = h | w \]
  where $v,u$ have the same color blue and none of $f_2, \ldots, f_{p-1}$
  have blue vertices. Similarly $v,w$ have the same color blue, and none of
  $f^\prime_2, \ldots, f^\prime_{q-1}$ have blue vertices.
  None of these paths is then a subpath of the other, and hence there is
  $r \geq 2$ such that $f_r \neq f_r^\prime$ and let $r$ be minimal such.
  If $r \geq 3$, as above \eqref{eq:partition-splice}, we get a walk 
  \[ w | h = f_q^\prime, \ldots, f_r^\prime, f^\prime_{r-1} = f_{r-1}, f_r,
    \ldots, f_p = g | u, \]
  where none of the intermediate facets have blue vertices. If $r \geq 2$,
  as above \eqref{eq:partition-splice2}, get a walk
  \begin{equation} 
    u | f_p, \ldots, f_2, f_2^\prime,
    \ldots, f_q^\prime | w.
    \end{equation}
Aagain as above these walks may be reduced to paths (cut-splice-reduction)
  and so $g  \sim^\prime_F h$. \qed
  \end{proof}

  \begin{lemma} \label{lem:partition-restrict}
Suppose the relation $\sim_F$,
   induces the relation $\sim_V$ on $V$. 
    Consider the subcomplex $X_m$. i) The restricted relation $\sim_F|_{X_m}$
    then induces the restricted relation $\sim_{V} |_{X_m}$.
    Similarly, ii)
    if $\sim_V$ induces $\sim_F$ the restricted relation $\sim_V|_{X_m}$
induces the restricted relation $\sim_{F} |_{X_m}$.
  \end{lemma}

  \begin{proof} Note that if $v,w \in V_m$ and $v \sim_V^\prime w$ and
    $v|f_1, \ldots, f_p | w$ is the path in $X$ between them,
    then since $X_m$ is stacked, this path is entirely in $X_m$.

\noindent i) Suppose given $\sim_F$. 
Let  $v,w \in V_m$ such that $v \sim_V w$, so we have
    \[ v = v_0 \sim^\prime_V v_1 \sim^\prime_V \cdots \sim^\prime_V v_t = w. \]
    Let $\ell$ minimal such that all $v_i \in V_\ell$. If $\ell > m$ then some
    $v_i \in V_\ell \backslash V_{\ell-1}$ where $0 < i < t$.
    By the above Lemma \ref{lem:partition-usimw} we may replace
    $v_{i-1} \sim^\prime_V v_i \sim^\prime_V v_{i+1}$ by $v_{i-1} \sim^\prime_V v_{i+1}$.
    In this way we may reduce the above so all $v_i \in V_m$, and thus
    $\sim_F |_{X_m}$ induces $\sim_V |_{X_m}$. 

    \noindent ii) Suppose given $\sim_V$.
    Let $f,g \in X_m$ and $f \sim_F g$, so we have
    \[ f = f_1 \sim^\prime_F f_2 \sim^\prime_F \cdots \sim^\prime_F = g. \]
    Again using Lemma \ref{lem:partition-gsimh}
    above we may reduce to all $f_i \in X_m$. \qed
    \end{proof}

\subsection{The main theorem}
  
\begin{theorem} 
  \label{thm:partition-VE}
  Let $X$ be a stacked simplicial complex of dimension $d$,
  with vertex set\/ $V$ and facet
set\/ $F$, and $s$ an integer $\geq 1$. 
The correspondences in Subsections \ref{subsec:partition-vf}
and \ref{subsec:partition-fv}
give a one-to-one correspondence between partitions of the vertices
$V$ into $r+d$ non-empty sets, each $s+1$-scattered,
and partitions of the facets $F$ into $r$ non-empty sets, each $s$-scattered. 
\end{theorem}

\begin{proof} {\it A.} Assume we have started from a partition
  $\sim_F$  of facets $F$, and have constructed the equivalence
  relation $\sim_V$ corresponding to a partition of the vertices $V$.
  We show properties {\it A1, A2, A3} for $X_m$ by induction on $m$,
  so that $\sim_V$ in turn induces the original partition $\sim_F$.

 \medskip

 \noindent {\it Property A1:} Suppose  $v \sim_V w$ are distinct and,
  say blue, and
  let $m$ be minimal such that $v,w \in X_m$. We may assume
  $v \in V_m \backslash V_{m-1}$. Suppose we have a path from $v$ to $w$:
  \[ v | f_1, \ldots, f_p | w \]
where $f_2, \ldots, f_{p-1}$ have
no blue vertices. We want to show $f_1 \sim_F f_p$ (in the original equivalence
relation for $F$), that is, they have the same color, say green.
Let $f_1$ have color green, and let
\begin{equation} \label{eq:partition-vw1}
  v = v_0 \sim^\prime_V v_1 \sim^\prime_V \cdots \sim^\prime_V v_t = w.
  \end{equation}
By Lemma \ref{lem:partition-restrict} we may assume
all $v_i \in X_m$.
Also, if $v_i \in V_m\backslash V_{m-1}$ for some $0 < i < t$, we
may by Lemma \ref{lem:partition-usimw} reduce to a shorter
such sequence. We may therefore assume the $v_i$ for $0 < i < t$
are in $V_{m-1}$. If $t = 1$, then $f_1 \sim_F f_p$ by definition of
$\sim_V^\prime$, and we are done.
So assume $t \geq 2$ and consider the path in $X_m$: 
\begin{equation} \label{eq:partition-v01}
  v = v_0 | f_1^\prime, \cdots, f_q^\prime | v_1.
\end{equation}
Then $f_1^\prime = f_1$ and $f_q^\prime$ have the same color green
by definition of $\sim_V^\prime$ and none of 
$f_2^\prime, \ldots, f^\prime_{q-1}$ are green.
Both $v = v_0$ and $v_1$ are blue. We claim that none of $f_2^\prime, \ldots,
f_{q-1}^\prime$ has any blue vertex.
Otherwise, let $2 \leq r \leq q-1$ be maximal such that $f^\prime_r$ has
a blue vertex $v^\prime$.
We get a sequence
$v^\prime | f^\prime_r, \ldots, f^\prime_q | v_1$.
Then  $v^\prime \in V_{m-1}$, since $g = f_1^\prime \backslash
\{v\} \sus V_{m-1}$ by Lemma \ref{lem:stacked-vVp},
and the path from $g$ to $v_1$ is in $V_{m-1}$). 
By induction (on $m$)
the facets $f_r^\prime$ and $f_q^\prime$ have the same color,
a contradiction.
Thus none of $f_2^\prime, \ldots, f_{q-1}^\prime$ has a blue vertex.

\medskip
We claim that $v$ and $w$ are independent, which is now equivalent to
show that $w$ is not in $f_1$. This will give $p \geq 2$.
Recall by
Lemma \ref{lem:stacked-vVp} that
$f_1 = f_v$ is the only facet in $X_m$ containing $v$.
If $w$ was in $f_1$
then $w \in f_1 \backslash \{v\} = f_1^\prime \backslash \{v\}$,
which is $f_1 \cap f_2 = f_1^\prime \cap f_2^\prime$. In particular $w \in V_{m-1}$.
By induction, since $w,v_1$ are in $X_{m-1}$ and are related by
\eqref{eq:partition-vw1}, they are either equal or independent.
If equal, $v$ and $w$ are independent since $v \sim_V^\prime w$, and so
not both in $f_1$. 
If $w$ and $v_1$ are independent, $w$ is not in $f_q^\prime$.
Then $q \geq 3$, and $f_2^\prime$
has a blue vertex $w$, contradicting that no intermediate facet in
\eqref{eq:partition-v01} has a blue vertex. The upshot is that $w$ is not
in $f_1$, and so $p \geq 2$. 


\medskip
By cut-splice-reducing the sequences, 
\begin{equation} \label{eq:partition-tosti}
  v | f_1, \ldots, f_p | w, \quad v|f_1^\prime, \ldots f_q^\prime | v_1,
  (\text{ where } f_1 = f_1^\prime),
\end{equation}
as in 
Lemma \ref{lem:partition-gsimh} we reduce to a path
 \begin{equation} \label{eq:partition-walk}
   v_1 | f_q^\prime, \ldots,  f_p | w, 
 \end{equation}
 where no intermediate facet has blue vertices. 

 \medskip
 \noindent {Case 1: $w \in V_{m-1}$.} Then by induction on $m$, since
 $v_1$ and $w$ are in $X_{m-1}$, $f_q^\prime$
 and $f_p$ have the same color green. As $f_q^\prime$ and $f_1^\prime = f_1$
 has the same color, green, we get that $f_1$ and $f_p$ are both green.

 \noindent {Case 2: $w \in V_m \backslash V_{m-1}.$}
 We have $w \sim_V v_1$, and only one of $w,v_1$ (i.e. $w$)
 is in $V_m \backslash V_{m-1}$. Then we can start the argument of Property
 A1 over again and reduce to Case 1.
 So we conclude again that $f_q^\prime$ and $f_p$
 have the same color. Again as $f_q^\prime$ and $f_1^\prime = f_1$ have
 the same color, green, we get that $f_1$ and $f_p$ are green.

 \medskip

\noindent {\it Property A2:} Suppose $f,g$ are distinct and green. Let
\[ f = f^1, f^2, \ldots, f^p = g \] be the unique path from $f$ to $g$.
Let $q > 1$ be minimal such that $f^q$ is green, and consider the path
\[ v | f^1, \ldots, f^q | w. \]
Then $v \sim_V w$ and so are, say blue. If one $f^r$ where $r \in [2,q-1]$ 
has a blue vertex, let $r$ be minimal such, and let $v^\prime$ be this
vertex, so we have a path
\[ v| f^1, \ldots, f^r | v^\prime. \] By part A1, $f^1$ and $f^r$ have the
same color, which must be green. This is a contradiction and so 
none of $f^2, \ldots, f^{q-1}$ has a blue vertex. Thus
$f^1 \sim_F^\prime f^q$. Let $f_0 = f$ and $f_1 = f^q$.
Considering now the shorter path $f^q, f^{q+1}, \ldots, f^p$. By
induction on length of path, there are
\[ f^q = f_1 \sim_F^\prime \cdots \sim_F^\prime f_r = f^p, \]
and so we get part A2.

\medskip
\noindent {\it Property A3:} Suppose $\sim_F$ is $s$-scattered.
Let $v,w$ be distinct blue vertices whose face-distance $p$ is as small as
possible. We showed in the argument of Property A1 that $v$ and $w$ are
independent, and so we have a path
\[ v | f^1, f^2, \ldots, f^p | w\]
with $p \geq 2$.
None of $f^2, \ldots, f^{p-1}$ can then have blue vertices.
Thus $f^1 \sim_F f^p$ by what we showed in A1, and their facet-distance
is $p-1 \geq s$. Whence $p \geq s+1$ and the blue vertices are $s+1$-scattered.

\medskip
\noindent {\it B.}
We have started from a partition $\sim_V$ of the vertices $V$
into independent sets.
We have constructed from this an equivalence relation $\sim_F$
and corresponding partition of the facets. 
We show that properties {\it B1, B2, B3} holds for $X_m$ by induction
on $m$, so that $\sim_F$ induces the original partition $\sim_V$.

\medskip
\noindent {\it Property B1:}
Suppose $f \sim_F g$, and the path from $f$ to $g$ is
\begin{equation} \label{eq:partition-vw2}
  v | f = f_1, \ldots, f_p = g | w
\end{equation}
where $f_1, f_p$ are green and none of $f_2, \ldots, f_{p-1}$ are green.
We show that $v \sim_V w$, they have the same color, say blue.

There is a sequence
\[ f = f^0 \sim^\prime_F f^1 \sim^\prime_F \ldots \sim^\prime_F f^t = g.
\]
Let $m$ be smallest such that $f,g$ is in $X_m$.
By Lemma \ref{lem:partition-restrict}
we may assume all $f^i$ are in $X_m$. If some $f^i$ for $0 < i < t$ is in
$X_m \backslash X_{m-1}$, by Lemma \ref{lem:partition-gsimh} we may reduce to
a shorter sequence. So we may assume $f^i \in X_{m-1}$ for $0 < i < t$. 

\medskip
If $t = 1$ then $v \sim_V w$ and we are done. So assume $t \geq 2$.
Since $v \in V_m \backslash V_{m-1}$, by Lemma \ref{lem:stacked-vVp}
we have $f^0 = f_v$.
Now look at at the path from $f^0$ to $f^1$
\[ v = v^0 | f^0 = f_1^{\prime }, \ldots, f_q^{\prime }= f^1 | v^1, \]
where $v = v^0$ and $v^1$ are blue,
and $f_2^\prime, \ldots, f_{q-1}^\prime$ do not have any blue vertex.
The facets $f_1^\prime = f^0 = f$ and $f_q^\prime$ have the same color, which is
green. Are there any green facets in between? Suppose $2 \leq r \leq q-1$
is maximal such that $f_r^{\prime }$ is green. We have a path 
\[ v^2 | f_r^{\prime }, \cdots, f_q^{\prime} | v^1 \]
where all $f_r^{\prime }, \ldots, f_q^{\prime }$ are in $X_{m-1}$.
By induction on $m$, 
$v^2$ and $v^1$ have the same color, blue. This is a contradiction,
as $f_r^\prime$ has no blue vertex.
So $f_1^{\prime }$ and $f_{q}^{\prime } $ are green, while no
facets in between are green.

\medskip
Look at the two paths: 
\[ v | f = f_1, \ldots, f_p = g | w, \quad v = v^0 | f_1 =
  f_1^\prime, \ldots, f_q^\prime | v^1 \]
where $v^1 \in V_{m-1}$.
None of these is a sub-sequence of the other, as $f_p$ and $f_q^\prime$
are green, and no intermediate facet is green. As in Lemma
\ref{lem:partition-gsimh} we may cut-splice-reduce these together to
get a path
\[ v^1 | f_q^\prime, \ldots, f_p | w ,\]
where only the end facets are green. 

\noindent {\it Case 1: $w \in V_{m-1}$.} Then in the path from
$v^1$ to $w$, the end facets are green, and no intermediate facet is
green. By induction on $m$ (since both $v^1$ and $w$ are in $X_{m-1}$),
we get that $v^1$ and $w$ have the same color.
Furthermore $v$ and $v^1$ have the same color, blue, so both $v$ and $w$
are blue.

\noindent {\it Case 2: $w \in V_{m}\backslash V_{m-1}$.}
We have exactly one of $w,v^1$ (i.e. $w$) in $V_m \backslash V_{m-1}$.
The path from $v^1$ to $w$ has end facets green and no
intermediate facet green.
But then we can start the argument of Property B1 over again, and
reduce to Case 1. So we conclude that $w$ and $v^1$ have the same color.
Since $v$ and $v^1$ are both blue, we get that $v$ and $w$ are
both blue.

\medskip
\noindent {\it Property B2:} Suppose $v,w$ are distinct and blue.
Let
\[ v | f^1, \ldots, f^p | w \] be the unique path from $v$ to $w$.
Since $v$ and $w$ are independent, $p \geq 2$. 
Let $q \geq 2$ be minimal such that $f^q$ contains a blue vertex
$v^\prime = v_1$. Then $f^1 \sim_F f^q$ by construction, say they are green.
Consider the path
\[ v | f^1, \ldots, f^q | v^\prime . \]
If one $f^r$ for $r \in [2,q-1]$ is green, let $r$ minimal such.
Then we have a path
\[ v | f^1, \ldots, f^r | v^{\prime \prime} \] and by B1 we
have $v \sim_V v^{\prime \prime}$, both blue. This contradicts
the choice of $q$. Thus $v \sim_V^\prime v^\prime$.
Now $v^\prime \in f^q$. If $q = p$ we have $v^\prime = w$ and we are done.
If $q < p$ then  $v^\prime \neq w$. Let $r$ be maximal with
$q \leq r < p$ such that $v^\prime \in f^{r}$. We then get
\[ v^\prime | f^{r}, \ldots, f^p | w\]
where both $v^\prime$ and $w$ are blue.

By induction on path
length there are
\[ v^\prime = v_1 \sim_V^\prime v_2 \sim_V^\prime \cdots \sim_V^\prime v_s = w,
\] and so we get part B2.

\medskip
\noindent {\it Property B3:} Suppose $\sim_V$ is $s+1$-scattered.
Let  $f \sim_F g$ be distinct green facets whose facet-distance
$p-1$ is as small as possible with path
\[ v | f= f^1, f^2, \ldots, f^p = g | w. \]
None of the intermediate facets $f^2, \ldots, f^{p-1}$ are green.
Then we have just shown in B1 that $v \sim_V w$ and so their
face-distance $p \geq s+1$. Then $p-1 \geq s$ and the green facets
are $s$-scattered.

\medskip
\noindent {\it Final part:} We show that if there are $r$ facet parts, there
are $r+d$ vertex parts. This is by induction on the number of facets.
Clearly this is true if we have one facet, a simplex.
For a stacked simplicial complex $X$ let $m$ be minimal such that
$X_m = X$, and let $X^\prime = X_{m-1}$. Let $f$ be a facet in
$X_m \backslash X_{m-1}$, and $v$ the free vertex of $f$.
By induction, if there are $r$ facet parts in $X^\prime$, there are
$r+d$ vertex parts.

If the facet $f$ makes a part of its own
in $X$, the free vertex $v$ becomes a part of its own, by construction
of vertex classes. Then we have $r+1$ facet parts and $r+1+d$ vertex parts
in $X$. 

If the facet $f$ is put into an existing part, say the green part,
look at paths $v | f= f_1, \ldots, f_p | w$ where $f_1, f_p$ are
green and the intermediate facets are not green. Then $v$ will be given
the color of $w$, say blue. If  $v | f= f_1^\prime, \ldots, f_q^\prime | w^\prime$
is another path with $f = f_1^\prime$ and $f_q^\prime$ green, and not
intermediate facet is green, by Lemma \ref{lem:partition-gsimh} we may
cut-splice-reduce and get in $X^\prime$ 
a path $w | f_p, \ldots, f_q^\prime | w^\prime$ with $f_p, f_q^\prime$ green
and with no intermediate green facets. Both $w$ and $w^\prime$ have the same
color. Then $v$ is uniquely in the blue color class. So $X$ has $r$ facet parts
and $r+d$ vertex parts. \qed
\end{proof}

\begin{corollary}
  Let $X$ be a tree, and $s \geq 1$.  There is a bijection between partitions of
  the vertices into $r+1$ non-empty parts, each part $s+1$-scattered, and
  partitions of the edges into $r$ non-empty parts, each $s$-scattered.
  \end{corollary}

  \begin{corollary}
    Let $X$ be a triangulation of a polygon and $s \geq 1$. There
    is a bijection between partitions of the vertices into $r+2$ non-empty
    parts, each part $s+1$-scattered, and partitions of the triangles into
    $r$ non-empty parts, each $s$-scattered.
  \end{corollary}

 \section{Partitions of natural numbers}
  \label{sec:nat}
The main theorem here appears quite surprising and non-trivial even
for the simplest of trees, the line graph. Then it corresponds to
studying ordered set partitions, which has a comprehensive theory
\cite{TM}.

Taking the colimit
of longer and longer line graphs, we get results for the natural numbers.
These are simple consequences
of known results for ordered set partitions \cite[Thm.2.2]{CDD} or
in a more enumerative form \cite[Thm.1]{CW}, but do not
seem to have been stated in this form for natural numbers.
In a similar vein, \cite{Mu} considers
partitions of intervals $[n]$, but with requirements on the parts which
are different from ours here.

\medskip
Consider the infinite line graph
\begin{center} \label{eq:natural-line}
\begin{tikzpicture}[dot/.style={draw,fill,circle,inner sep=1pt},scale=1.2]
\draw [help lines, white] (1,-.5) grid (6,0);
\draw[very thick] (0,0)--(1,0);
\draw[very thick] (1,0)--(2,0);
\draw[very thick] (2,0)--(3,0);
\draw[very thick] (3,0)--(4,0);
\draw[very thick] (4,0)--(5,0);
\draw[dashed] (5,0)--(7,0);
\fill (0,0)  circle (0.09);
\fill (1,0)  circle (0.09);
\fill (2,0)  circle (0.09);
\fill (3,0)  circle (0.09);
\fill (4,0)  circle (0.09);
\fill (5,0)  circle (0.09);
\end{tikzpicture}
\end{center}

The set of edges $E$ may be identified with the natural numbers $\NN$, and
also the set of vertices $V$ may be identified with $\NN$.
Let $L_n$ be the line graph with $n$ edges. The bijection between
partitions edges into $r$ parts, each $d$-scattered, and
partitions of vertices into $r+1$ parts, each $d$-scattered,
is compatible with extending the line graph $L_n = (V_n,E_n)$ with one edge and
vertex to the line graph $L_{n+1}$: If $(E^i)$ is a partition
of edges in $L_{n+1}$ corresponding to a partition $(V^j)$ of vertices.
Then the partition $(E^i \cap E_n)$ corresponds to $(V^j \cap V_n)$.

Thus taking the colimit, we get for the infinite line graph
\eqref{eq:natural-line} a bijection between partitions
of edges $(E^i)$ into $r$ parts, each $d$-scattered, 
and partitions of vertices $(V^j)$ into $r+1$ parts, each $d+1$-scattered.

\medskip
Recall that a subset $A \sus \NN$ of natural numbers
is $d$-scattered if for every $p < q$ in $A$ we have $q-p \geq d$. 


\begin{theorem} \label{thm:nat-rd}
  There is a bijection between partitions of $\NN$ into $r$ sets, each
  $d$-scattered, and partitions of $\NN$ into $r+1$ sets, each
  $d+1$-scattered.
\end{theorem}

By successively going from $r$ to $r+1$ to $r+2$ and so on we get:

\begin{corollary}
There is a bijection between partitions of $\NN$ into $r+1$ parts
$A_0 \sqcup A_1 \sqcup \cdots \sqcup A_{r}$ (each set by default being
$1$-scattered) and partitions of $\NN$ into
$r+d$ parts $B_0 \sqcup \cdots \sqcup B_{r+d-1}$, each $d$-scattered.
\end{corollary}

Specializing to $r = 0$ we get the trivial fact that there is
a unique partition of $\NN$ into $d$ parts, each $d$-scattered.
Clearly this partition is the congruence classes modulo $d$.
However specializing to $r = 1$, we get the quite non-trivial:

\begin{corollary} \label{cor:nat-d} For each $d \geq 1$
there is a bijection between partitions of $\NN$ into two parts
$A_0 \sqcup A_1$ and partitions of $\NN$ into
$d+1$ parts $B_0 \sqcup \cdots \sqcup B_{d}$, each being $d$-scattered.
\end{corollary}

\begin{example}
  For $r = 1$ let the partition be $\{p\} \sqcup \NN \backslash \{p \}$,
  the one part consisting of a single element $p$.
  This corresponds to a partition of $\NN$ into three parts, each
  being $2$-scattered. We use $\ldots$ to indicate arithmetic progression
  of step size $2$. These parts are:
  \begin{align*}
     \ldots, p-4, p-2, &\,\, p \\
    & \,\, p+1, p+3, p+5, \ldots \\
    \ldots, p-3, p-1,& \, \, p+2, p+4, \ldots \\ 
  \end{align*}

  It also corresponds to a partition of $\NN$ into four parts, each
  being $3$-scattered. (Here $\ldots$ denotes progression with
  step size $3$.) These parts are:
  \begin{align*}
     \ldots, p-6, p-3, &\, \,  p \\
  \ldots, p-5, p-2,  & \, \, p+1, p+4, p+7, \ldots \\
                       & \, \, p+2, p+5, \ldots \\
    \ldots, p-4, p-1, & \, \, p+3, p+6, \ldots .
  \end{align*}
\end{example}

\begin{example}
  Let again $r=1$. Consider the partition
  $\{p, q \} \sqcup \NN \backslash \{p,q\}$
  where $p < q$. It corresponds to the following three parts, each
  $2$-scattered. We get two cases, according to whether $q-p$ is
  even or odd. 
  When $q-p$ is even:
  \begin{align*}
    \ldots, p-3, & \, \, p-1, p+2, \quad \ldots  &  q-2, &\, \,  q \\
    \ldots, p-4, p-2, & \, \, p, &  & \, \, q+1, q+3,\ldots \\
   & \, \,  p+1, p+3,  \quad\ldots & q-3, q-1, & \, \, q+2, q+4, \ldots .
  \end{align*}
  Note that in the middle part there is quite a long gap from
  $p$ to $q+1$. 
  When $q-p$ is odd:
 \begin{align*}
    & \, \, p+1, p+3, \quad \ldots  &  q-2, & \, \, q \\
    \ldots, p-4, p-2, & \, \, p, &  & \, \, q+1, q+3, \ldots \\
    \ldots, p-3, & \, \, p-1, p+2, \quad \ldots & q-1, & \, \, q+2, q+4, \ldots 
\end{align*}          
\end{example}

\bibliographystyle{amsplain}      
\bibliography{biblio}

\end{document}